\documentclass[11pt]{amsart}%
\usepackage{amsfonts}
\usepackage{amssymb}
\usepackage{amsmath}
\usepackage{graphicx}%
\usepackage{color}
\usepackage{amscd}
\usepackage[all,cmtip]{xy}
\newtheorem{theorem}{Theorem}[section]
\newtheorem{lemma}{Lemma}[section]

\theoremstyle{definition}

\theoremstyle{remark}
\newtheorem{remark}[theorem]{Remark}

\numberwithin{equation}{section}

\newcommand{\abs}[1]{\lvert#1\rvert}

\newcommand{\C}{\mathbb{C}}
\newcommand{\Z}{\mathbb{Z}}
\newcommand{\N}{\mathbb{N}}
\newcommand{\ov}{\overline}
\newcommand{\begeq}{\begin{equation}}
\newcommand{\stopeq}{\end{equation}}
\newcommand{\ep}{\epsilon}

\newcommand{\dis}{\displaystyle}
\newcommand{\pa}{\partial}
\newcommand{\ei}[1]{\mathrm{e}^{#1}}

\newcommand{\Om}{\Omega}

\newcommand{\ta}{\theta}

\newcommand{\begar}{\begin{array}}
\newcommand{\stopar}{\end{array}}

\newcommand{\norm}[1]{\left\lVert#1\right\rVert}

\begin{document}
\title[Singular CR equation]
{A Generalized CR equation with isolated singularities.}
\author{B. De Lessa Victor}
\thanks{\small The first author was financed in part by the Coordenação de Aperfei\c coamento de Pessoal de N\'ivel Superior - Brasil (CAPES) - Finance Code 001 and in part by FAPESP (grant number 2021/03199-9) }
\address{\small Departamento de Matem\'atica, Instituto de Ciências Matemáticas e de Computação (ICMC), Universidade de São Paulo (USP), São Carlos, SP, Brazil.}
\email{brunodelessa@gmail.com}

\author{Abdelhamid Meziani}
\address{\small Department of Mathematics, Florida International University, Miami, FL, 33199, USA}
\email{meziani@fiu.edu}



\begin{abstract}
The generalized CR equation
$u_{\ov{z}}=au+b\ov{u}+f$ is studied when the coefficients $a$ and $b$ have a finite number of
singular points inside the domain. Solutions are constructed via the study of an associated integral
operator and the existence of nontrivial solutions of the associated homogeneous equation is established.
\end{abstract}
\maketitle

\section{Introduction}
The study of generalized CR equations
\[
\frac{\pa u}{\pa\ov{z}}=a(z)u+b(z)\ov{u}+f(z)
\]
in a domain $\Om\subset\C$ was initiated by L. Bers and I.N. Vekua
in \cite{Bers} and \cite{Vek}. 
This equation is of fundamental importance and has applications in many areas
(see for example \cite{Kra} and \cite{Rod} and the references therein).
The initiators of the theory considered the elliptic case when the coefficients are in $L^p(\Om)$ with $p>2$
and this situation is now well understood (see \cite{Beg} for a comprehensive presentation).
The case of degenerate coefficients (either on the boundary of the domain or inside the domain) is
of current interest.
Of particular interest to us, and in view of application to the study of deformation of surfaces
\cite{Bru-Mez}, we consider equations involving a finite number of isolated singular points.
Such type of equations were considered in \cite{Beg-Dai},\cite{Mez1}, \cite{Mez2},
\cite{Ras-SolB}, \cite{Ras-Sol},  \cite{Usm1}, \cite{Usm2}.

In this paper we consider the equation
\[
\frac{\pa u}{\pa\ov{z}}=\frac{A(z)}{L(z)}u+\frac{B(z)}{L(z)}\ov{u}+F(z),
\]
where $\dis L(z)=\prod_{j=1}^N(z-z_j)$ and $z_1,\cdots, z_N$ are distinct points in the domain
$\Om$.  It should be noted that  the case $N=1$ is studied  in \cite{Mez1}, through
the use of associated systems of ordinary differential equations when the coefficients $A,\ B$ depend only on the argument $\theta$ of $z$ and in \cite{Mez2} when the coefficients could also depend on $|z|$
but with small norm. The main results, Theorems \ref{Theo1} and \ref{Theo2} describe the solutions of such equations.
To prove Theorem \ref{Theo1},  we make use of and   associated integral operator \eqref{TLm1} and its adjoint with respect to a real bilinear form, which is inspired by the
recent result in \cite{Ras-Sol} by A.B. Rasulov and A.P. Soldatov  when the case of a single singular point and
small coefficients is considered. Theorem \ref{Theo2} shows the existence of nontrivial solutions
for the homogeneous equation ($F=0$).

\section{Main Results}

Let $\Om\subset\C$  be a relatively compact domain, $S=\{ z_1,\,\cdots ,\, z_N\}$ be a collection of $N$ distinct points in $\Om$ and
$A(z),\, B(z)\, \in L^\infty(\Om)\cap C^\infty(\ov{\Om}\backslash S)$.
Assume that for every $j\in \{1,\,\cdots ,\, N\}$, there exist
$0<\tau_j<1$, $\ \delta_j >0$, $\ 2\pi$-periodic functions $p_j(\ta)$, $q_j(\ta)$ and functions $A_j(r\ei{i\ta})$, $ B_j(r\ei{i\ta})$
such that $A_j,\, B_j,\, \in L^\infty (D(0,\delta_j))\cap C^\infty (D(0,\delta_j)\backslash\{0\})$,
where  $D(0,\delta_j)$ denotes the open disc centered at $0$ and with 
radius $\delta_j$, and
\begin{equation}\label{AB-condition}
\begar{ll}
\dis A(z_j+r\ei{i\ta}) & = \dis p_j(\ta)+r^{\tau_j}A_j(r\ei{i\ta}) \\
\dis B(z_j+r\ei{i\ta}) & = \dis q_j(\ta)+r^{\tau_j}B_j(r\ei{i\ta})\, .
\stopar
\end{equation}

Consider
\begin{equation}\label{gammajML}
\begar{l}
\dis \gamma_j=\frac{1}{\pi}\int_0^{2\pi}\!\!\ei{-2i\ta}p_j(\ta)\, d\ta,  \ \ j=1,\cdots \, N,\\
 \dis L(z)=\prod_{j=1}^N(z-z_j)\quad\text{and}\quad
M(z)=\prod_{j=1}^N|z-z_j|^{\gamma_j}.
\stopar
\end{equation}
 Our goal is to understand the solutions
of the equation
\begin{equation}\label{CRequation1}
\frac{\pa u}{\pa\ov{z}}=\frac{A(z)}{L(z)}u+\frac{B(z)}{L(z)}\ov{u}+F(z),
\end{equation}
when the nonhomogeneous term $F(z)$ vanishes at the set of singular points $\left\{z_1, \ldots, z_N\right\}$.

For positive numbers $m$ and $p$, with $m\in\Z^+$ and $p>1$, consider the Banach space
\[
E_{m,p}=\left\{ f:\Om\,\longrightarrow\,\C\, :\ \frac{f(z)}{L(z)^m}\,\in\, L^p(\Om)\right\}
\]
equipped with the norm
\[
\norm{f}_{m,p}=\norm{\frac{f(z)}{L(z)^m}}_{L^p(\Om)}\, .
\]
The main results of this paper are the following theorems.

\begin{theorem}\label{Theo1}
Let $A$ and $B$ be functions satisfying (\ref{AB-condition}), $m\in\Z^+$ and $p>2$.
Then for every function $F$ in $\Om$ such that $\dis\frac{F(z)}{M(z)}\in E_{m+1,p}(\Om)$,
where $M(z)$ is given in (\ref{gammajML}), there exists a function
$v\in E_{m,p}(\Om)\cap C^\alpha (\Om\backslash S)$, with $\dis \alpha =(p-2)/p$ such that
the function $u(z)=v(z)M(z)$ is a solution of the equation (\ref{CRequation1}).
Moreover, if in addition $F\in C^{k,\sigma}(\Om\backslash S)$ with $k\in\Z^+$ and $0<\sigma<1$, then
$u\in C^{k+1,\sigma}(\Om\backslash S)$.
\end{theorem}

\begin{remark}
In the paper \cite{Ras-Sol}, equation (\ref{CRequation1}) is studied in the presence of a single singular point $p_0$ (so $N=1$) and when the coefficient $B$ has small norm. In this case, the authors prove the existence of solutions of the form $v(z)/|z-p_0|^a$ with $a<1$.
 In our case, we only require the number of singular points to be finite and there is no restriction on the size of the norms of the coefficients.
\end{remark}

\begin{theorem}\label{Theo2}
Let $A$ and $B$ be functions satisfying (\ref{AB-condition}) and $k\in\Z^+$.
The homogeneous equation
\begin{equation}\label{CRequationh}
\frac{\pa u}{\pa\ov{z}}=\frac{A(z)}{L(z)}u+\frac{B(z)}{L(z)}\ov{u}
\end{equation}
has non trivial solutions in $C^k(\Om)$. Moreover, for any $a>0$, a nontrivial
solution
$u$ can be chosen so that $u$ vanishes to an order $\ge a$ at each
singular point $z_j$.
\end{theorem}

The rest of the paper deals with the proof of these results.

\section{Reduction to the case $A=0$}

In this section we show that the solvability of equation (\ref{CRequation1})
can be reduced to an analogous equation where the coefficient $A=0$.
For this we start by proving the following lemma.

\begin{lemma}\label{functionw}
For $j=1,\,\cdots ,\, N$ let $\gamma_j$ be as in (\ref{gammajML}).
Then there exists a function $\mu\in L^\infty(\Om)\cap C^\infty(\Om\backslash S)$ such
that
\begin{equation}\label{w}
w(z)=\sum_{j=1}^N\gamma_j\log|z-z_j|\, +\mu(z)
\end{equation}
satisfies
\begin{equation}\label{w-equation}
\frac{\pa w(z)}{\pa\ov{z}}=\frac{A(z)}{L(z)}.
\end{equation}
\end{lemma}

\begin{proof}
Let $\delta >0$ be such that  the discs $D(z_j,2\delta)$, with $j=1, \cdots, N$, are contained in $\Om$ and are pairwise disjoint. Let $\phi_1,\,\cdots ,\, \phi_N\in C^\infty(\C)$ such that
\begin{equation*}
\text{$\phi_j\equiv 1$ in the disc $D(z_j,\delta)$, \ \  $\textrm{Supp}(\phi_j)\subset D(z_j,2\delta)$}
\end{equation*}
and set $\dis \phi_0=1-\sum_{j=1}^N\phi_j$. Note that $\phi_0\equiv 1$ in
$\dis \Om\backslash (\bigcup_{j=1}^ND(z_j,2\delta))$ and
$\phi_0\equiv 0$ in
$\dis \bigcup_{j=1}^ND(z_j,\delta)$.
The solvability of the equation (\ref{w-equation}) can be reduced to those of the $N+1$ equations
\begin{equation}\label{wj-equation}
\frac{\pa w_j}{\pa\ov{z}} =\frac{A\phi_j}{L},\quad j=0,\, \dots ,\, N
\end{equation}
and taking $\dis w=\sum_{j=0}^Nw_j$.  Note that since $\phi_0\equiv 0$ in
$\dis \bigcup_{j=1}^ND(z_j,\delta)$, it follows that $\dis\frac{A\phi_0}{L}\in C^\infty(\ov{\Om})$
and so for $j=0$, equation (\ref{wj-equation}) has a solution
$w_0\in C^\infty(\ov{\Om})$ (see \cite{Beg}).

For $j=1,\,\cdots ,\, N$, we use polar coordinates around the point $z_j$, that is, set
$z=z_j+r\ei{i\theta}$,  and use property (\ref{AB-condition}) of the function $A$
   to transform equation (\ref{wj-equation})
into an equation of the form
\begin{equation}\label{wj-polarequation}
\frac{\pa w_j}{\pa r}+\frac{i}{r}\frac{\pa w_j}{\pa\theta} =\frac{\gamma_j +\hat{p_j}(\ta)}{r}
+r^{\tau_j-1}c_j(r,\ta)\, ,
\end{equation}
where $\hat{p_j}(\ta)$ is a $2\pi$-periodic, $C^\infty$ function with zero average and
$c_j(r,\ta)$ is a bounded function, $C^\infty$ for $r>0$. Since $r^{\tau_j-1}c_j\in L^p(\overline{\Omega})$ with $\dis 2<p<\frac {2}{1-\tau_j}$, then  equation
\[
\frac{\pa v_j}{\pa r}+\frac{i}{r}\frac{\pa v_j}{\pa\theta} =r^{\tau_j-1}c_j(r,\ta)\, ,
\]
has a solution $v_j\in C^\alpha(\Om)\cap C^\infty(\ov{\Om}\backslash\{z_j\})$ with $\alpha =(p-2)/p$
(see \cite{Beg}).

The function
\[
\zeta_j(r,\ta)=\gamma_j\log r -i\int_0^\ta \hat{p_j}(s)\, ds
\]
satisfies
\[
\frac{\pa \zeta_j}{\pa r}+\frac{i}{r}\frac{\pa \zeta_j}{\pa\theta} =\frac{\gamma_j +\hat{p_j}(\ta)}{r}\, .
\]
It follows that the function
\[
w_j(r,\ta)=\zeta_j(r,\ta)+v_j(r,\ta)=\gamma_j\log r+ \underbrace{\left(v_j(r,\ta)-i\int_0^\ta \hat{p_j}(s)\, ds\right)}_{:= \mu_{j}}
\]
solves equation (\ref{wj-equation}).
Therefore
\[
w(z)=\sum_{j=1}^N \gamma_{j} \log|z-z_j|+\, \mu(z),
\]
with $\dis \mu(z)=w_0(z)+ \sum_{j=1}^N \mu_{j}(z)$ is the desired
solution of equation \eqref{w-equation}.
\end{proof}

With $w(z)$ given by  in Lemma (\ref{functionw}), the function
\[
\ei{w(z)}=\left(\prod_{j=1}^N|z-z_j|^{\gamma_j}\right)\, \ei{\mu(z)}=M(z)\, \ei{\mu(z)}
\]
is smooth in $\ov{\Om}\backslash S$. A function $u(z)$ satisfies equation (\ref{CRequation1}) if and
only if the function
\[
v(z)=\ei{-w(z)}u(z)=\ei{-\mu(z)}\, \frac{u(z)}{M(z)}
\]
solves the equation
\begin{equation}\label{CRequation2}
\frac{\pa v}{\pa\ov{z}}=\frac{B_1(z)}{L(z)}\ov{v}\, +F_1(z)\, ,
\end{equation}
with
\[
B_1(z)  = B(z)\ei{\ov{w(z)}-w(z)}\ \ \textrm{and}\ \
F_1(z)  = \frac{\ei{-\mu(z)}F(z)}{M(z)}\, .
\]
Note $|B_1(z)|=|B(z)|$ and that $F_1\in E_{m+1,p}(\Om)$ if and only if $\dis\frac{F}{M}\in E_{m+1,p}(\Om)$.
Thanks to this reduction, from now on we will assume that $A=0$ and consider the equation
\begin{equation}\label{CRequation3}
\frac{\pa u}{\pa\ov{z}}=\frac{B(z)}{L(z)}\ov{u}\, +F(z)\, .
\end{equation}

\section{Properties of an associated integral operator}

For $L(z)$ as given in (\ref{gammajML}) and $m\in \Z^+$, consider the integral operator
$T_{L,m}$ defined by
\begin{equation}\label{TLm1}
T_{L,m}u(z)=\frac{-L(z)^m}{\pi}\int_\Om\frac{B(\zeta)\ov{u(\zeta)}}{L(\zeta)^{m+1}(\zeta-z)}\, d\xi d\eta
\end{equation}
where $\zeta=\xi+i\eta$. We have the following lemma.

\begin{lemma}\label{lemma1TLm}
For $p>2$, the operator
$\dis T_{L,m}:\, E_{m+1,p}(\Om)\, \longrightarrow\, E_{m,p}(\Om)$ is bounded
and $\dis T_{L,m}:\, E_{m+1,p}(\Om)\, \longrightarrow\, C^0(\Om)$ is compact.
Furthermore
\[
T_{L,m}\left(E_{m+1,p}(\Om)\right)\, \subset\, C^\alpha(\Om), \quad\text{for}\quad
\alpha =\frac{p-2}{p}.
\]
\end{lemma}

\begin{proof}
The boundedness of $T_{L,m}$ is a consequence of estimates for the classical Cauchy-Pompeiu operator.
Indeed, for $u\in E_{m+1,p}(\Om)$ with $p>2$, we have
\[\begar{ll}
\dis\abs{T_{L,m}u(z)} &\dis \le \frac{\abs{L(z)}^m}{\pi}
\int_\Om\frac{\abs{B(\zeta)}\, \abs{u(\zeta)}}{\abs{L(\zeta)}^{m+1}\abs{\zeta-z}}\, d\xi d\eta\\ \\
&\dis \le C\abs{L(z)}^m\, \norm{B}_\infty \,\norm{u}_{m+1,p},
\stopar
\]
where $C$ is a constant depending only on $p$ and the size of the domain $\Om$.
From this it also follows that $T_{L,m}u\, \in E_{m,p}(\Om)$.

Now for arbitrary $z_1,\, z_2$ distinct points in $\Om$ , we have
\begin{equation}\label{TLm2}
T_{L,m}u(z_1)-T_{L,m}u(z_2)=
\frac{-1}{\pi}\int_\Om\frac{B(\zeta)\ov{u(\zeta)}}{L(\zeta)^{m+1}}\,
\frac{\zeta K_1(z_1,z_2) +K_2(z_1,z_2)}{(\zeta-z_1)(\zeta-z_2)}\, d\xi d\eta,
\end{equation}
where
\small{
\begin{equation}\label{K1K2}
\begin{gathered}
K_1(z_1,z_2) = L(z_1)^m-L(z_2)^m =(z_1-z_2)P_m(z_1,z_2),\\
K_2(z_1,z_2) = L(z_1)^m(z_1-z_2)-z_1(L(z_1)^m-L(z_2)^m) =(z_1-z_2)Q_m(z_1,z_2),
\end{gathered}
\end{equation}
}
\normalsize
where $P_m$ and $Q_m$ are polynomials in $z_1, z_2$ of  degrees $Nm-1$ and $Nm$, respectively.
It follows from (\ref{TLm2}) and (\ref{K1K2})
\small{
\begin{equation}
\begin{split}
\abs{T_{L,m}u(z_1)-T_{L,m}u(z_2)} &\dis \le
C\norm{B}_\infty \abs{z_1-z_2}\int_\Om\!\!\!\frac{\abs{u(\zeta)}\, d\xi d\eta}{\abs{L(\zeta)}^{m+1}\abs{\zeta-z_1}\abs{\zeta-z_2}}\\
&\le  \dis C\norm{B}_\infty \abs{z_1-z_2}\,\norm{u}_{m+1,p}\left(\int_\Om\frac{ d\xi d\eta}{\abs{\abs{\zeta-z_1}^q\abs{\zeta-z_2}^q}}\right)^{1/q}\\
& \le C'\norm{B}_\infty\,\norm{u}_{m+1,p} \abs{z_1-z_2}^{(p-2)/p},
\end{split}
\end{equation}
}
\normalsize
\noindent where the constants depend only on $p$ and $\Om$ and $q$ . In the last estimate, we used Hadamard's inequality (see \cite{Beg} or \cite{Vek}).
The compactness of $T_{L,m}$ follows from the compactness of embedding into H\"{o}lder spaces.

\end{proof}

Next we define the adjoint of $T_{L,m}$  with respect to the real bilinear form $\langle .,.\rangle$ given by
by
\[
\langle \phi,\psi\rangle =\text{Re}(\phi,\psi)=\text{Re}\left(\int_\Om \phi(z)\ov{\psi(z)}\, dx dy\right)\, .
\]
If $q$ is the H\"{o}lder conjugate of $p$, we set
\[
T^\ast_{L,m}:\, X_{m,q}(\Om)\, \longrightarrow\, X_{m+1,q}(\Om),
\]
defined in the space
\[
X_{m,q}(\Om)=\{ v:\, \Om\,\longrightarrow\, C:\  L(z)^mv(z)\in L^q(\Om)\}\, ,
\]
 and given by
\begin{equation}\label{Tconjugate}
T^\ast_{L,m}v(\zeta )=\frac{-B(\zeta)}{\pi L(\zeta)^{m+1}}\int_\Om\!\!\frac{L(z)^m\ov{v(z)}}{z-\zeta}\, dx dy\, .
\end{equation}

\begin{remark}
The fact that we can take $X_{m+1,q}(\Om)$ as the target space is a consequence of Theorem 1.2.6 of
\cite{Vek}.
\end{remark}

Now consider the operator
\[
P_{L,m}\,: E_{m+1,p}(\Om)\,\longrightarrow\, E_{m,p}(\Om),\ \ P_{L,m}u:= u- T_{L,m}u,
\]
and its adjoint
\[
P^\ast_{L,m}\,: X_{m,q}(\Om)\,\longrightarrow\, X_{m+1,q}(\Om),\ \ P^\ast_{L,m}v:= v- T^\ast_{L,m}v.
\]
We have the following lemmas

\begin{lemma}\label{KerPLm}
The spaces $\textrm{Ker}(P_{L,m})$ and $\textrm{Ker}(P^\ast_{L,m})$ are finite dimensional.
Consequently, the operators $P_{L,m}$ and $P^\ast_{L,m}$ are Fredholm.
\end{lemma}

\begin{proof}
First note that for $u\in \textrm{Ker}(P_{L,m})$  we have $u=T_{L,m}u$ and since $p>2$, then
$\dis u\in C^\alpha (\Om)\subset L^2(\Om)$ (where $\alpha =(p-2)/p$). We continue the proof
by contradiction, suppose that $\textrm{dim}\left(\textrm{Ker}(P_{L,m})\right)=\infty$,
then we can find an $L^2$-orthonormal basis $\{u_j\}_{j \in \N}$. It follows from
the compactness of $T_{L,m}$ and from the fact that $u_j=T_{L,m}u_j$ that $\{u_j\}_{j \in \N}$ has a convergent subsequence.
This is absurd, since $\norm{u_j-u_k}_2=\sqrt{2}$ for $j\ne k$.

Now consider $v\in \textrm{Ker}(P^\ast_{L,m})$; then
\begin{equation}\label{vinKerPstar}
v(\zeta)=T^\ast_{L,m}v(\zeta)= -
\frac{B(\zeta)}{\pi L(\zeta)^{m+1}}\int_\Om\!\!\frac{L(z)^m\ov{v(z)}}{z-\zeta}\, dx dy\, .
\end{equation}
Since $L(z)^m\ov{v(z)}\,\in L^q(\Om)$ (with $q>1$),  it follows (Theorem 1.26 of \cite{Vek}) that
$v\in X_{m+1,2}(\Om)$. With this, a similar argument as the one used for $\textrm{Ker}(P_{L,m})$
shows that $\textrm{Ker}(P^\ast_{L,m})$ is also finite dimensional.
\end{proof}

\begin{lemma}\label{EmpDecomposition}
Let $H_{m,p}(\Om)=E_{m,p}(\Om)\cap\mathcal{H}(\Om)$, where $\mathcal{H}(\Om)$ denotes the
space of holomorphic functions in $\Om$. Then
\[
E_{m,p}(\Om)=\textrm{Range}(P_{L,m})\, +\, H_{m,p}(\Om).
\]
\end{lemma}

\begin{proof}
To prove this result, it is enough to verify that if $v\in X_{m,q}(\Om)$ is such that
$v\in \textrm{Range}(P_{L,m})^\perp \cap H_{m,p}(\Om)^\perp$, then $v=0$.
Since
\begin{equation*}
\textrm{Range}(P_{L,m})^\perp =\textrm{Ker}(P^\ast_{L,m}),
\end{equation*}
such a function $v$
satisfies (\ref{vinKerPstar}). Also it follows from $L(z)^m\ov{v(z)}\in L^q(\Om)$, with
$1<q<2$, that
\[
\int_\Om\frac{L(z)^m\ov{v(z)}}{z-\zeta}dx dy \, \in L^\gamma(\Om),
\]
for any $\dis q<\gamma<\frac{2q}{2-q}$ (see \cite{Beg} or \cite{Vek} for properties
of the Cauchy-Pompeiu operator).
By repeating this argument, we find that in fact $v$ is H\"{o}lder continuous in
$\Om\backslash S$.

The function $\dis h_\zeta(z)=\frac{L(z)^m}{z-\zeta}$ is holomorphic in $\Om$ for any
$\zeta\notin\Om$. Moreover, since $h_\zeta$ vanishes to order $m$ on the set $S$,
then  $h_\zeta \in E_{m,p}(\Om)$ for $\zeta\notin\ov{\Om}$. Thus $h_\zeta\in H_{m,p}(\Om)$
for $\zeta\notin\ov{\Om}$. It follows from $v\in H_{m,p}(\Om)^\perp$ that
\[
\langle h_\zeta ,v \rangle=\textrm{Re}\left(\int_\Om\!\! \frac{L(z)^m\ov{v(z)}}{z-\zeta}\, dx dy\right) =0.
\]
Similarly $\langle ih_\zeta ,v \rangle=0$. This means that $(h_\zeta ,v)=0$, which implies that
$T^\ast_{L,m}v(\zeta)=0$ for every $\zeta\notin\ov{\Om}$.
Hence $v(\zeta)$ is extends continuously to $\C\backslash S$ by taking it to be zero outside $\Om$.
In particular $v=0$ on the boundary $\pa\Om$.

Let $Z=\{ z\in\Om: \ B(z)=0\}$ and $\Om_1=\Om\backslash Z$.  It follows from
(\ref{vinKerPstar}) that $v=0$ on the closed set $Z$. Since $v=0$ on $\pa\Om$, then
$v=0$ on $\pa\Om_1$. By differentiating  (\ref{vinKerPstar}) with respect to $\ov{\zeta}$, we
get
\[
\frac{\pa v}{\pa \ov{\zeta}}=\frac{1}{B(\zeta)}\frac{\pa B(\zeta)}{\pa\ov{\zeta}}\, v -
\frac{B(\zeta)}{L(\zeta)}\, \ov{v}
\]
in $\Om_1\backslash S$. This equation is elliptic and it follows that locally $v$ is similar to
a holomorphic function. Since $v=0$ on $\pa\Om_1$, then $v=0$ everywhere. This completes the proof
of the lemma.
\end{proof}

\section{Proof of Theorem \ref{Theo1}}

Let $\{w_1,\,\cdots ,\, w_n\}\,\subset\, X_{m,q}(\Om)$ be a basis of $\textrm{Ker}(P^\ast_{L,m})$.
It follows from Lemma \ref{EmpDecomposition} and from the fact that the operator $P_{L,m}$ is
Fredholm that we can find $\{h_1,\,\cdots ,\, h_n\}\, \subset\, H_{m,p}(\Om)$ such that
$\langle w_j,h_k \rangle=\delta_{jk}$, where $\delta_{jk}$ is the Kronecker symbol, and such that
\begin{equation}\label{EmpDecomposition2}
E_{m,p}(\Om)=\textrm{Range}(P_{L,m})  + \textrm{Span}\{h_1,\,\cdots ,\, h_n \}\, .
\end{equation}
Now given $F\in E_{m,p}(\Om)$, define
\begin{equation}\label{functionf}
f(z)=\frac{-L(z)^m}{\pi}\int_\Om\!\! \frac{F(\zeta)}{L(\zeta)^m (\zeta -z)}\, d\xi d\eta\, .
\end{equation}
Then $f\in E_{m,p}(\Om)$ and $\dis\frac{\pa f}{\pa\ov{z}}=F$.  The decomposition
(\ref{EmpDecomposition2}) implies the existence of real constants $c_1,\, \cdots ,\, c_n$
such that
\[
f=g+\sum_{k=1}^nc_k h_k\, ,
\]
with $g\in \textrm{Range}(P_{L,m})$. Let $u\in E_{m+1,p}(\Om)$ such that
$P_{L,m}u=g$. Thus,
\[
P_{L,m}u=u-T_{L,m} u= g=f-\sum_{k=1}^nc_kh_k.
\]
It follows from the definition of the operator $T_{L,m}$ in (\ref{TLm1}) and from
the fact that $h_1,\, \cdots ,\, h_n$ are holomorphic that
\[
\frac{\pa u}{\pa\ov{z}} =\frac{\pa}{\pa\ov{z}}\left(  f-\sum_{k=1}^nc_kh_k +T_{L,m}u\right)
=F(z)+\frac{B(z)}{L(z)}\, \ov{u}.
\]
Therefore the function $u$ solves (\ref{CRequation3}). This completes the
proof of the Theorem \ref{Theo1}.

\begin{remark}\label{vanishingorder}
It follows from the definition of the function $f$ given in (\ref{functionf}) that if the function
$F$ vanishes to infinite order at the singular points $z_1,\,\cdots ,\, z_N$, then the solution
$u$ can be taken to vanish to any prescribed order at the points $z_j$. For it is enough to
take $m$ larger than the prescribed order of vanishing.
\end{remark}

\section{Proof of Theorem \ref{Theo2}}

To show the existence of nontrivial solutions in $\Om$ of the homogeneous equation
\begin{equation}\label{CRhomogeneous}
\frac{\pa u}{\pa\ov{z}}=\frac{B(z)}{L(z)}\,\ov{u}\, ,
\end{equation}
we first recall results from \cite{Mez1} and \cite{Mez2} dealing with particular
cases of (\ref{CRhomogeneous}).

Let $q(\ta)$ be a $2\pi$-periodic and $C^\infty$ function. It is proved in \cite{Mez1}
that there exist a sequence of positive numbers
\[
0<\lambda_1<\lambda_2<\,\cdots \,<\lambda_n<\,\cdots\quad\textrm{with}
\quad \lim_{n\to\infty}\lambda_n=\infty
\]
and a sequence of nonvanishing  $2\pi$-periodic and $C^\infty$ functions $f_k(\ta)$ such
for every $k\in \Z^+$, the function $v_k(r,\ta)=r^{\lambda_k}f_k(\ta)$ solves the equation
\begin{equation}\label{CRhomogeneousA}
\frac{\pa v}{\pa\ov{z}}=\frac{q(\ta)}{r}\,\ov{v}\,  \ \text{where} \ z=r\ei{i\ta}\, .
\end{equation}

Let $R>0$ and $g\in C^\infty(D(0,R)\backslash\{0\})$ such that $g=O(r^\alpha)$ for some
$0<\alpha <1$. Then the proof of Theorem (4.1) of \cite{Mez2} shows that every solution of
the equation
\begin{equation}\label{CRhomogeneousB}
\frac{\pa v}{\pa\ov{z}}=\frac{q(\ta)+g(r,\ta)}{r}\,\ov{v}
\end{equation}
is similar to a solution of (\ref{CRhomogeneousA}) and vice versa.
This means that for every $v_k(r,\ta)$ solution of (\ref{CRhomogeneousA})
there exists a bounded function $s_k(r,\ta)$ in the disc $D(0,R)$ such that the function
$w_k=v_k\ei{s_k}$ solves (\ref{CRhomogeneousB}).

Now we turn back to equation (\ref{CRhomogeneous}) in the domain $\Om$.
It follows from hypothesis \eqref{AB-condition} that the function $B$ is of
the form $q(\ta)+g(r,\ta)$ with $g=O(r^{\tau_j})$ in a neighborhood of the singular
point $z_j$ (here the polar coordinates are centered at $z_j$). It follows
that equation (\ref{CRhomogeneous}) has a nontrivial solution $v_j$ defined in a neighborhood
of $z_j$ and moreover, $v_j$ can be chosen to vanish to any prescribed order at $z_j$. Thus
such a function is class $C^k$ near $z_j$ (provided that the order of vanishing at $z_j$ is
large enough).

Let $\ep >0$ be such that $D(z_j,2\ep)\subset\Om$ for every $j=1,\cdots, N$ and the discs are pairwise disjoint.  For any fixed $k \in \N$ and each $j$,  let
$\phi_j\in C^{k}(\C)$  with $\textrm{Supp}(\phi_j)\subset D(z_j,2\ep)$ and
$\phi_j\equiv 1$ in the disc $D(z_j,\ep)$. We assume that $\ep$ is small enough so that the
functions $v_j$ (described above) are defined in the disc $D(z_j,2\ep)$.
We are going to construct a solution $u$ of (\ref{CRhomogeneous}) in the whole domain $\Om$
of the form
\begin{equation}\label{uw}
u(z)=w(z)+\sum_{j=1}^N\phi_j(z)v_j(z)\, .
\end{equation}

In order for the function $u$ to satisfy (\ref{CRhomogeneous}), the function $w$ needs to
solve
\begin{equation}\label{CRforw}
\frac{\pa w}{\pa\ov{z}}=\frac{B(z)}{L(z)}\,\ov{w} +F(z)\, ,
\end{equation}
with
\begin{equation}\label{functionF}
F(z)=\sum_{j=1}^N\left(\frac{B(z)}{L(z)}\ov{\phi_j(z)}\ov{v_j(z)} -
\frac{\pa(\phi_j(z)v_j(z))}{\pa\ov{z}}\right)\, .
\end{equation}
Note that $F\equiv 0$ in the discs $D(z_j,\ep)$ for $j=1,\cdots, N$.
It follows from Theorem \ref{Theo1} that equation (\ref{CRforw}) has a solution
and that  $w\in C^{k+1}(\Om\backslash S)$ and that $w$ can be chosen to vanish to any prescribed order
at the singular points (Remark \ref{vanishingorder}).

Now we need to verify that for some set of cut off functions $\{\phi_1,\, \cdots ,\, \phi_N\}$
the constructed solution $u$ is not trivial.
By contradiction, suppose that for every $\{\phi_1,\, \cdots ,\, \phi_N\}$, $u\equiv 0$.
The corresponding function $w$ satisfies
\[\begar{lll}
w & =-\dis\sum_{j=1}^Nv_j &\quad\textrm{in}\ \ \bigcup_{j=1}^ND(z_j,\ep),\\
w & =0 &\quad\textrm{in}\ \ \Om_1=\Om\backslash\left(\bigcup_{j=1}^ND(z_j,\ep)\right)\, .
\stopar\]
Let $u'(z)=w'(z)+\sum_{j=1}^N\phi'_j(z)v_j(z) \equiv 0$ be another such solution corresponding to another
set $\{\phi'_1,\, \cdots ,\, \phi'_N\}$ of cut off functions. Consider
\begin{equation*}
\text{$\phi'_2=\phi_2,\, \cdots ,\, \phi'_N=\phi_N$ and $\phi'_1\ne \phi_1$.}
\end{equation*}

Set $\psi=\phi'_1-\phi_1$ and $W=w'-w$. Hence
$W\equiv 0$ everywhere except possibly on the annulus $A_\ep =D(z_1,2\ep)\backslash D(z_1,\ep)$
and it satisfies the equation
\begin{equation}\label{Wequation}
\frac{\pa W}{\pa\ov{z}} =\frac{B(z)}{L(z)}\,\ov{W} +G_\psi(z),
\end{equation}
where
\begin{equation}\label{functionG}
G_\psi (z)=\ov{\psi}(z)\,\frac{\pa v_1(z)}{\pa\ov{z}}
-\frac{\pa (\psi(z) v_1(z))}{\pa\ov{z}}=
\ov{\psi}(z)\,\frac{B(z)\ov{ v_1(z)}}{L(z)}
-\frac{\pa (\psi(z) v_1(z))}{\pa\ov{z}}.
\end{equation}

Let $p_0\in A_\ep$ such that $\dis B(p_0)\ne 0$ and $v_1(p_0)\ne 0$, and so $\dis\frac{\pa v_1}{\pa\ov{z}}(p_0)\ne 0$,
and let $\delta>0$ small enough so that $\dis\frac{\pa v_1}{\pa\ov{z}}(z)\ne 0$ for all $z\in D(p_0,\delta)$
(the same is again true for $B$ and $v_1$).
We can assume after translation that $p_0=0$. In particular, we are in a situation where for any function
\begin{equation*}
\psi =\phi'_1-\phi_1\, \in C^{k}(\ov{D(0,\delta)}) \ \text{with $\textrm{Supp}(\psi)\subset\ov{D(0,\delta)}$},
\end{equation*}
 any solution of the equation (\ref{Wequation})
satisfies $W(z)=0$ for all
$z\in\C$ with $|z|\ge\delta$. We are going to show that such a situation cannot happen.

We use results from \cite{Vek}, Chapter III, Sections 8 and 10,  dealing with representations
of generalized analytic functions. In our case, we apply such representation to the solutions of equation
 (\ref{Wequation}) in the disc $\ov{D(0,\delta)}$. Thus, there exist kernels $K_1(\zeta,z)$ and
 $K_2(\zeta,z)$ (formulas 8.16, page 168 of \cite{Vek}) depending only on the coefficient $B/L$ in the disc
 such that
 \begin{equation}\label{kerK1K2}
 K_1(\zeta,z)=\frac{1}{\zeta -z} +C_1(\zeta,z)\quad\textrm{and}\quad K_2(\zeta,z)= C_2(\zeta,z),
 \end{equation}
 where $C_1,\, C_2$ are $C^\infty$ functions (because $B/L$ is $C^\infty$ in the annulus $A_\ep$)
 such that any solution $W$ of (\ref{Wequation}) has the representation
 \begin{equation}\label{Wrepresentation}
W(z)=\frac{-1}{\pi}\int_{D(0,\delta)}\!\!
\left(K_1(\zeta,z)G_\psi(\zeta)+K_2(\zeta,z)\ov{G_\psi(\zeta)}\right)\, d\xi d\eta \, .
\end{equation}
Note that  the original formula of Vekua contains an additional term involving an integral over the
boundary of the domain. But in our case the additional term is 0 because $W=0$ outside $D(0,\delta)$.
It follows from (\ref{Wrepresentation}) that
\begin{equation}\label{Integral=0}
\int_{D(0,\delta)}\!\!\left(K_1(\zeta,z)G_\psi(\zeta)+K_2(\zeta,z)\ov{G_\psi(\zeta)}\right)\, d\xi d\eta =0\quad
\forall z\notin D(0,\delta)\, ,
\end{equation}
and this relation holds for any function $\psi$ as above. We use (\ref{kerK1K2}) to rewrite (\ref{Integral=0}) as
\begin{equation}\label{Integral2}
\int_{D(0,\delta)}\!\!\frac{G_\psi(\zeta)d\xi d\eta}{z-\zeta}=
\int_{D(0,\delta)}\!\!\left(C_1(\zeta,z)G_\psi(\zeta)+C_2(\zeta,z)\ov{G_\psi(\zeta)}\right)\, d\xi d\eta \, ,
\end{equation}
for $|z|\ge \delta$. By  using the expression of $G_\psi$
given in (\ref{functionG}) we have
\begin{equation*}
\int_{D(0,\delta)}\!\!\frac{G_\psi(\zeta)d\xi d\eta}{z-\zeta}= - \int_{D(0,\delta)}\!\!
\frac{\ov{\psi}(\zeta)\,B(\zeta)\ov{ v_1(\zeta)}}{L(\zeta)(z-\zeta)}\, d\xi d\eta\, .
\end{equation*}
Now we select the function $\psi$ as
\[
\psi(\zeta)=(\delta^2-r^2)^{2k}\frac{\ov{L(\zeta)}}{\ov{B(\zeta)}\, v_1(\zeta)}\quad\textrm{for}
\ \zeta=r\ei{i\ta}\, \in D(0,\delta)\, .
\]
It follows that
\[
G_\psi(\zeta)=(\delta^2-r^2)^{2k}-\frac{\pa}{\pa\ov{\zeta}}\left({(\delta^2-r^2)^{2k}}\frac{\ov{L(\zeta)}}{\ov{B(\zeta)}}\right)
\]
and so
\begin{equation}\label{Gestimate}
\abs{G_\psi(\zeta)}\le C_{k}\delta^{4k-1},\quad \forall \zeta\in D(0,\delta),
\end{equation}
where $C_{k}$ is a constant which does not depend on $\delta$.
For such a choice of the function $\psi$ the relation \eqref{Integral2} becomes
\begin{equation}\label{Integral3}
\int_{D(0,\delta)}\!\!\frac{(\delta^2-r^2)^{2k}}{z-\zeta}\, d\xi d\eta=
\int_{D(0,\delta)}\!\!\left(C_1(\zeta,z)G_\psi(\zeta)+C_2(\zeta,z)\ov{G_\psi(\zeta)}\right)\, d\xi d\eta \, ,
\end{equation}
for $|z|\ge\delta$.
We evaluate the left hand side of (\ref{Integral3}) by using the series expansion
$\dis\frac{1}{z-\zeta}=\sum_{j=0}^\infty\frac{\zeta^j}{z^{j+1}}$, for $|z|>|\zeta|$, and by using polar coordinates to integrate. 
We find
\[
\frac{\pi \delta^{4k + 2}}{(2k+1)z}=
\int_{D(0,\delta)}\!\!\left(C_1(\zeta,z)G_\psi(\zeta)+C_2(\zeta,z)\ov{G_\psi(\zeta)}\right)\, d\xi d\eta \, ,
\]
for $|z|>\delta$.  By differentiating  with respect to $z$ and using \eqref{Gestimate}, we find the estimate
\begin{align*}
\frac{\pi \delta^{4k + 2}}{(2k+1)|z|^{2}}\ &\le
\int_{D(0,\delta)}\!\!\left(\left|{\frac{\pa C_1(\zeta,z)}{\pa z}}\right|
+\left|\frac{\pa C_2(\zeta,z)}{\pa z}\right|\right)\, \abs{G_\psi(\zeta)}
\, d\xi d\eta \,  \\
&\le C_{k} \delta^{4k-1} \int_{D(0,\delta)}\!\!\left(\left|{\frac{\pa C_1(\zeta,z)}{\pa z}}\right|
+\left|\frac{\pa C_2(\zeta,z)}{\pa z}\right|\right)\, 
\, d\xi d\eta \,.
\end{align*}

Since the functions inside the integral are bounded, by possibly increasing $C_{k}$ we have
\[
\frac{\pi \delta^{4k + 2}}{(2k+1)|z|^{2}}\le \pi C_{k} \delta^{4k+1},
\]
where $C_{k}$ does not depend on $\delta$ when it is small. 
If we take $|z| =  \displaystyle\frac{3 \delta}{2}$, it follows that
\begin{equation*} 
\frac{\pi \delta^{4k + 2}}{(2k+1)}\le \pi C_{k} \delta^{4k+3} \ \ \Rightarrow \ 1 \leq (2k+1) C_{k} \delta. 
\end{equation*}
This is clearly a contradiction if we reduce $\delta$ sufficiently, which completes the proof of the theorem.

\vskip 0.1in

\noindent{\bf Acknowledgements.} Part of this work was done when the first author was visiting the Department of Mathematics \& Statistics at FIU (Florida International
University). He would like to thank the members of the institution for the support provided during his visit.

\end{document}